\documentclass[11pt]{article}
\usepackage{fullpage,amsmath,amsthm,amsfonts,amssymb}
\usepackage{url}
\newcommand{\Erdos}{Erd\H{o}s }
\newcommand{\Erdosns}{Erd\H{o}s}
\newcommand{\alphap}{{\alpha'}}
\newcommand{\betap}{{\beta'}}

\newcommand{\ang}[1]{\langle#1\rangle}

\newcommand{\Z}{{\sf Z}}
\newcommand{\nat}{{\sf N}}
\newcommand{\rat}{{\sf Q}}
\newcommand{\real}{{\sf R}}

\newcommand{\support}{\textup{supp}}
\newcommand{\sg}{\textup{sig}}

\newcommand{\xvec}[1]{\ifcase 3{#1} {\ang {a_1,a_2,a_3} } \else 
\ifcase 4{#1} {\ang{a_1,a_2,a_3,a_4}} \else {\ang {a_1,\ldots,a_{#1}}}\fi\fi}
\newcommand{\yvec}[1]{\ifcase 3{#1} {\ang {y_1,y_2,y_3} } \else 
\ifcase 4{#1} {\ang{y_1,y_2,y_3,y_4}} \else {\ang {y_1,\ldots,y_{#1}}}\fi\fi}
\newcommand{\zvec}[1]{\ifcase 3{#1} {\ang {z_1,z_2,z_3} } \else 
\ifcase 4{#1} {\ang{z_1,z_2,z_3,z_4}} \else {\ang {z_1,\ldots,z_{#1}}}\fi\fi}
\newcommand{\vecc}[2]{\ifcase 3{#2} {\ang { {#1}_1,{#1}_2,{#1}_3 } } \else
\ifcase 4{#1} {\ang { {#1}_1,{#1}_2,{#1}_3,{#1}_{4} } }
\else {\ang { {#1}_1,\ldots,{#1}_{#2}}}\fi\fi}
\newcommand{\veccd}[3]{\ifcase 3{#2} {\ang { {#1}_{{#3}1},{#1}_{{#3}2},{#1}_{{#3}3} } } \else
\ifcase 4{#1} {\ang { {#1}_{{#3}1},{#1}_{{#3}2},{#1}_{#3}3},{#1}_{{#3}4} }
\else {\ang { {#1}_{{#3}1},\ldots,{#1}_{{#3}{#2}}}}\fi\fi}
\newcommand{\veccz}[2]{\ifcase 3{#2} {\ang { {#1}_0,{#1}_2,{#1}_3 } } \else
\ifcase 4{#1} {\ang { {#1}_0,{#1}_2,{#1}_3,{#1}_{4} } }
\else {\ang { {#1}_0,\ldots,{#1}_{#2}}}\fi\fi}
\newcommand{\xve}[1]{\ifcase 3{#1} {a_1,a_2,a_3} \else 
\ifcase 4{#1} {a_1,a_2,a_3,a_4} \else {a_1,\ldots,a_{#1}}\fi\fi}
\newcommand{\yve}[1]{\ifcase 3{#1} {y_1,y_2,y_3} \else 
\ifcase 4{#1} {y_1,y_2,y_3,y_4} \else {y_1,\ldots,y_{#1}}\fi\fi}
\newcommand{\zve}[1]{\ifcase 3{#1} {z_1,z_2,z_3} \else 
\ifcase 4{#1} {z_1,z_2,z_3,z_4} \else {z_1,\ldots,z_{#1}}\fi\fi}
\newcommand{\ve}[2]{\ifcase 3#2 {{#1}_1,{#1}_2,{#1}_3} \else
\ifcase 4#2 {{#1}_1,{#1}_2,{#1}_3,{#1}_{4}}
\else {{#1}_1,\ldots,{#1}_{#2}}\fi\fi}
\newcommand{\ved}[3]{\ifcase 3#2 {{#1}_{{#3}1},{#1}_{{#3}2},{#1}_{{#3}3}} \else
\ifcase 4#2 {{#1}_{{#3}1},{#1}_{{#3}2},{#1}_{{#3}3},{#1}_{{#3}4}}
\else {{#1}_{{#3}1},\ldots,{#1}_{{#3}{#2}}}\fi\fi}
\newcommand{\fuve}[3]{
\ifcase 3#2
{{#3}({#1}_1),{#3}({#1}_2,{#3}({#1}_3)} \else
\ifcase 4#2
{{#3}({#1}_1),{#3}({#1}_2),{#3}({#1}_3),{#3}({#1}_4)}
\else
{{#3}({#1}_1),\ldots,{#3}({#1}_{#2})}\fi\fi}
\newcommand{\setmathchar}[1]{\ifmmode#1\else$#1$\fi}
\newcommand{\vlist}[2]{%
	\setmathchar{%
		\compound#2\one{#2}\two
		\ifcompound
			({#1}_1,\ldots,{#1}_{#2})
		\else
			\ifcat N#2
				({#1}_1,\ldots,{#1}_{#2})
			\else
				\ifcase#2
					({#1}_0)\or
					({#1}_1)\or
					({#1}_1,{#1}_2)\or 
					({#1}_1,{#1}_2,{#1}_3)\or
					({#1}_1,{#1}_2,{#1}_3,{#1}_4)\else 
					({#1}_1,\ldots,{#1}_{#2})
				\fi
			\fi
		\fi}}

\newif\ifcompound
\def\compound#1\one#2\two{%
	\def\one{#1}
	\def\two{#2}
	\if\one\two
		\compoundfalse
	\else
		\compoundtrue
	\fi}

\newcommand{\xwe}[1]{\ifcase 3{#1} {a_1\wedge a_2\wedge a_3} \else 
\ifcase 4{#1} {a_1\wedge a_2\wedge a_3\wedge a_4} \else {a_1\wedge \cdots \wedge
a_{#1}}\fi\fi}
\newcommand{\we}[2]{\ifcase 3#2 {\ang { {#1}_1\wedge {#1}_2\wedge {#1}_3 } } \else
\ifcase 4{#1} {\ang { {#1}_1\wedge {#1}_2\wedge {#1}_3\wedge {#1}_{4} } }
\else {\ang { {#1}_1\wedge \cdots\wedge {#1}_{#2}}}\fi\fi}

\newcommand{\st}{\mathrel{:}}

\newcommand{\into}{\rightarrow}

\newcommand{\es}{\emptyset}

\newcommand{\CL}{\mathord{\mbox{\it CL}}}
\newcommand{\COL}{\mathord{\mbox{\it COL}}}

\newcommand{\union}{\cup}

\newcommand{\s}[1]{\s_{#1}}

\newcommand{\monus}{\;\raise.5ex\hbox{{${\buildrel
    \ldotp\over{\hbox to 6pt{\hrulefill}}}$}}\;}

\newcounter{savenumi}

\newtheorem{theoremfoo}{Theorem}[section] 
\newenvironment{theorem}{\pagebreak[1]\begin{theoremfoo}}{\end{theoremfoo}}

\newtheorem{lemmafoo}[theoremfoo]{Lemma}
\newenvironment{lemma}{\pagebreak[1]\begin{lemmafoo}}{\end{lemmafoo}}
\newtheorem{conjecturefoo}[theoremfoo]{Conjecture}

\newtheorem{conventionfoo}[theoremfoo]{Convention}
\newenvironment{convention}{\pagebreak[1]\begin{conventionfoo}\rm}{\end{conventionfoo}}

\newtheorem{corollaryfoo}[theoremfoo]{Corollary}
\newenvironment{corollary}{\pagebreak[1]\begin{corollaryfoo}}{\end{corollaryfoo}}

\newtheorem{exercisefoo}[theoremfoo]{Exercise}

\newtheorem{openfoo}[theoremfoo]{Open Problem}

\newcommand{\fig}[1] 
{
 \begin{figure}
 \begin{center}
 \input{#1}
 \end{center}
 \end{figure}
}

\newtheorem{potanafoo}[theoremfoo]{Potential Analogue}

\newtheorem{notefoo}[theoremfoo]{Note}
\newenvironment{note}{\pagebreak[1]\begin{notefoo}\rm}{\end{notefoo}}

\newtheorem{nttn}[theoremfoo]{Notation}
\newenvironment{notation}{\pagebreak[1]\begin{nttn}\rm}{\end{nttn}}

\newtheorem{examfoo}[theoremfoo]{Example}

\newtheorem{dfntn}[theoremfoo]{Definition}
\newenvironment{definition}{\pagebreak[1]\begin{dfntn}\rm}{\end{dfntn}}

\newtheorem{propositionfoo}[theoremfoo]{Proposition}

\newenvironment{sketch}{\begin{proof}[Proof sketch]}{\end{proof}}

\newcommand{\yyskip}{\penalty-50\vskip 5pt plus 3pt minus 2pt}
\newcommand{\blackslug}{\hbox{\hskip 1pt
        \vrule width 4pt height 8pt depth 1.5pt\hskip 1pt}}
\newcommand{\QED}{{\penalty10000\parindent 0pt\penalty10000
        \hskip 8 pt\nolinebreak\blackslug\hfill\lower 8.5pt\null}
        \par\yyskip\pagebreak[1]}

\newtheorem{factfoo}[theoremfoo]{Fact}
\newenvironment{fact}{\pagebreak[1]\begin{factfoo}}{\end{factfoo}}

\newenvironment{block}{\begin{list}{\hbox{}}{\leftmargin 1em
    \itemindent -1em \topsep 0pt \itemsep 0pt \partopsep 0pt}}{\end{list}}

\begin{document}

\centerline{\bf A Statement in Combinatorics that is}
\centerline{\bf Independent of ZFC (An Exposition)}
\centerline{\bf by Stephen Fenner\footnote{fenner@cse.sc.edu} and William Gasarch\footnote{gasarch@cs.umd.edu}}

\begin{abstract}
It is known that, for any finite coloring of $\nat$, there exists
distinct naturals $e_1,e_2,e_3,e_4$ that are the same color such that
$e_1+e_2=e_3+e_4$. Consider the following statement which we denote $S$:
{\it For every $\aleph_0$-coloring of the reals 
there exists distinct reals $e_1,e_2,e_3,e_4$ such that
$e_1+e_2=e_3+e_4$?} Is it true? \Erdos showed that $S$ is equivalent
to the negation of the Continuum Hypothesis, and hence $S$ is indepedent
of ZFC. We give an exposition of his proof and some modern observations
about results of this sort.
\end{abstract}

\section{Introduction}

There are some statements that are independent of
Zermelo-Frankl Set Theory (henceforth ZFC).
Such statements cannot be proven or disproven
by conventional mathematics.  
The Continuum Hypothesis is one such statement
(``There is no cardinality strictly between $\aleph_0$ and $2^{\aleph_0}$.'')
Many such statements are unnatural in that they deal with objects only set theorists
and other logicians care about.

We present a natural statement in combinatorics that is
independent of ZFC.  
The result is due to \Erdosns.
In the last section we will discuss the question of whether the statement
is really natural.

\begin{notation}
We use $\nat$ to denote $\{0,1,2,\ldots\}$.
We use $\nat^+$ to denote $\{1,2,3,\ldots\}$.
If $n\in \nat^+$ then $[n]$ is the set $\{1,2,\ldots,n\}$.
We use $\real$  to denote the sets of real numbers.
We use $\Z$ to denote the integers.
We use $k$-AP to refer to an arithmetic progression with $k$ distinct elements.
For a set $A$ and $k\in\nat$, we let $\binom{A}{k}$ denote the set of $k$-element subsets of $A$.
\end{notation}

\begin{convention}
A set $A$ is {\it countable} if it is finite or there is a bijection of $A$
to $\nat$.
\end{convention}

\section{Colorings and Equations}

\begin{definition}
A {\it finite coloring} of a set $S$ is a map from $S$ to a finite set.
An {\it $\aleph_0$-coloring} of a set $S$ is a map from $S$ to a countable set.
\end{definition}

The following theorem is well known. We prove it for the sake of completeness.

\begin{theorem}\label{th:four}
For any finite coloring of $\nat^+$,
there exists distinct monochromatic $e_1,e_2,e_3,e_4$ such that
$$e_1 + e_2 = e_3 +e_4.$$
\end{theorem}

\begin{proof}
Let $\COL$ be a finite coloring of $\nat^+$.
Let $[c]$ be the image of $\COL$.

\noindent
{\bf First Proof}

Recall Ramsey's theorem~\cite{ramseynotes,GRS,Ramsey} 
on $\nat$: for any finite coloring of {\it unordered pairs of naturals}
there exists an infinite set $A$ such that all pairs of elements from $A$ have
the same color.

Let $\COL^*: \binom{\nat}{2} \into  [c]$ be defined by $\COL^*(\{a,b\}) = \COL(|a-b|)$.
Let $A$ be the set that exists by Ramsey's theorem.
Let $a_1<a_2<a_3<a_4\in A$.
Since $A$ is infinite we can take $a_1,a_2,a_3,a_4$ such that the six numbers
$\left\{a_j-a_i \st \{i,j\} \in \binom{[4]}{2}\right\}$ are distinct.

Since all of the $\COL^*(\{a_i,a_j\})$ are the same color
we have that, for $i<j$, $\COL(a_j-a_i)$ are all the same color.
Let
\begin{align*}
e_1 &= a_2-a_1 \\
e_2 &= a_4-a_2 \\
e_3 &= a_3-a_1 \\
e_4 &= a_4-a_3
\end{align*}

Clearly $e_1,e_2,e_3,e_4$ are distinct,
$\COL(e_1)=\COL(e_2)=\COL(e_3)=\COL(e_4)$, and $e_1+e_2=e_3+e_4$.

\noindent
{\bf Second Proof}

Recall van der Waerden's theorem~\cite{VDWbook,GRS,RamseyInts,VDW}:
For all $k$, for any finite coloring of $\nat^+$, there exists a monochromatic $k$-AP, that is, a $k$-AP all of whose elements are the same color.

Apply van der Waerden's Theorem to $\COL$ with $k=4$.
There exists $a,d\in \nat^+$ such that
$a$, $a+d$, $a+2d$, $a+3d$ are the same color.
Let
\begin{align*}
e_1 &= a \\
e_2 &= a +4d \\
e_3 &= a +2d \\
e_4 &= a +3d
\end{align*}
\end{proof}

\begin{note}~
Rado's theorem characterizes which equations lead to theorems like Theorem~\ref{th:four}
and which ones do not. We will discuss Rado's theorem in Section~\ref{se:moref}.
\end{note}

\section{What If We Color the Reals?}

What if we finitely color the Reals?
Theorem~\ref{th:four}
will still hold since we can just
restrict the coloring to $\nat^+$.
What if we $\aleph_0$-color the reals?

Let $S$ be the following statement:

\bigskip

{\it For any $\aleph_0$-coloring of the reals, there exist distinct monochromatic $e_1,e_2,e_3,e_4$ such that}
$$e_1 + e_2 = e_3 +e_4.$$

\bigskip

Is $S$ true?
This is the wrong question.
It turns out that $S$  is equivalent to the negation of CH, and hence
is independent of ZFC.
Komj\'{a}th~\cite{komjath} claims that \Erdos proved this result.
The proof we give is due to Davies~\cite{radoch}.
The goal of our paper is to present and popularize this result.
For more references on these types of results see Vega's paper~\cite{vega}.

\begin{definition}
$\omega$ is the first infinite ordinal, namely $\{1<2<3<\cdots \}$.
(Formally it is any ordering that is equivalent to $\{1<2<3<\cdots \}$.)
$\omega_1$ is the first uncountable ordinal.
$\omega_2$ is the first ordinal with cardinality bigger than $\omega_1$.
\end{definition}

\begin{fact}\label{fa:omega}\ 
\begin{enumerate}
\item
Assume CH.
Then there is a bijection between $\real$ and $\omega_1$.  
For all $\alpha\in\omega_1$ let $\alpha$ map to $x_{\alpha}$.
We can picture the reals listed out as such:
$$x_0,x_1,x_2, \ldots, x_\alpha, \ldots.$$
Note that, for all $\alpha\in \omega_1$, 
the set $\{ x_\beta \mid \beta <\alpha \}$ is countable.
\item
Assume $\neg$CH. Then there is an injection from $\omega_2$ to $\real$.
\end{enumerate}
\end{fact}

\section{CH $\implies \neg S$ }

\begin{definition}
Let $X\subseteq \real$.
Then $\CL(X)$ is the smallest set $Y\supseteq X$ that is closed under
addition, subtraction, multiplication, and division.
\end{definition}

\begin{lemma}\label{le:cl}~
\begin{enumerate}
\item
If $X$ is countable then $\CL(X)$ is countable.
\item
If $X_1 \subseteq X_2$ then $\CL(X_1)\subseteq \CL(X_2)$.
\end{enumerate}
\end{lemma}

\begin{proof}

\noindent
1) Assume $X$ is countable.
$\CL(X)$ can be defined with an $\omega$-induction
(that is, an induction just through $\omega$).
\begin{align*}
C_0 &= X \\
C_{n+1} &= C_n\union\{ a+b,a-b,ab\mid a,b \in C_n \} \union \{ a/b \mid a,b \in C_n, b\ne 0 \}\\
\end{align*}
One can easily show that
$\CL(X)=\union_{i=0}^\infty C_i$ and that this set is countable.

\noindent
2) This is an easy exercise.
\end{proof}

\begin{theorem}\label{th:false}
Assume CH. There exists an
$\aleph_0$-coloring of $\real$ such that there are no distinct monochromatic $e_1,e_2,e_3,e_4$ such that
$$e_1+e_2=e_3+e_4.$$
\end{theorem}

\begin{proof}

Since we are assuming CH, we have, by Fact~\ref{fa:omega}.1,
a bijection between $\real$ and $\omega_1$. For each $\alpha\in\omega_1$ let $x_\alpha$
be the real that $\alpha$ maps to.


For $\alpha < \omega_1$ let
$$X_\alpha = \{ x_\beta \mid \beta < \alpha \}.$$

Note the following:
\begin{enumerate}
\item
For all $\alpha$, $X_\alpha$ is countable.
\item
$
X_0\subset X_1 \subset X_2 \subset X_3 \subset \cdots \subset X_\alpha \subset\cdots
$
\item
$\bigcup_{\alpha<\omega_1} X_\alpha =  \real.$
\end{enumerate}

We define another increasing sequence of sets $Y_\alpha$ by letting
$$Y_\alpha = \CL(X_\alpha).$$

Note the following:
\begin{enumerate}
\item
For all $\alpha$, $Y_\alpha$ is countable.  This is from Lemma~\ref{le:cl}.1.
\item
$
Y_0\subseteq Y_1 \subseteq Y_2 \subseteq Y_3 \subseteq \cdots \subseteq Y_\alpha \subseteq\cdots\,$.
This is from Lemma~\ref{le:cl}.2.
\item
$\bigcup_{\alpha<\omega_1} Y_\alpha =  \real.$
\end{enumerate}

We now define our last sequence of sets:

For all $\alpha<\omega_1$,
$$Z_\alpha = Y_\alpha - \left(\bigcup_{\beta<\alpha} Y_\beta\right).$$

Note the following:
\begin{enumerate}
\item
Each $Z_\alpha$ is countable.
\item
The $Z_\alpha$ form a partition of $\real$ (although some of the $Z_\alpha$ may be empty).
\end{enumerate}

We will now define an $\aleph_0$-coloring of $\real$:
For each $\alpha<\omega_1$ we color $Z_\alpha$ with colors in $\omega$
making sure that every element of $Z_\alpha$ has a different color
(this is possible since $Z_\alpha$ is at most countable).

Assume, by way of contradiction, that there are distinct monochromatic $e_1,e_2,e_3,e_4$ such that
$$e_1+e_2=e_3+e_4.$$

Let $\alpha_1,\alpha_2,\alpha_3,\alpha_4\in\omega_1$ be such that
$e_i \in Z_{\alpha_i}$.
Since all of the elements in any $Z_\alpha$ are colored differently,
all of the $\alpha_i$'s are different.
We will assume $\alpha_1<\alpha_2<\alpha_3<\alpha_4$.
The other cases are similar.
Note that
$$e_4=e_1+e_2-e_3.$$
and
$$e_1,e_2,e_3 \in Z_{\alpha_1} \cup Z_{\alpha_2} \cup Z_{\alpha_3}\subseteq Y_{\alpha_1} \cup Y_{\alpha_2} \cup Y_{\alpha_3} = Y_{\alpha_3}.$$

Since $Y_{\alpha_3}=\CL(X_{\alpha_3})$
and $e_1,e_2,e_3\in Y_{\alpha_3}$, we have
$e_4 \in Y_{\alpha_3}$.
Hence $e_4\notin Z_{\alpha_4}$. This is a contradiction.
\end{proof}

What was it about the equation
$$e_1+e_2 = e_3 +e_4$$
that made the proof of Theorem~\ref{th:false} work?
Absolutely nothing:

\begin{theorem}\label{th:gen}
Let $n\ge 2$.
Let $a_1,\ldots,a_n\in \real$ be nonzero.
Assume CH. There exists an
$\aleph_0$-coloring of $\real$ such that there are no distinct monochromatic
$e_1,\ldots,e_n$ such that
$$\sum_{i=1}^n a_ie_i = 0.$$
\end{theorem}

\begin{sketch}
Since this proof is similar to the last one we just sketch it.

Let $X_\alpha$, $Y_\alpha$,  $Z_\alpha$ be defined as in Theorem~\ref{th:false}.
With these definitions define an $\aleph_0$-coloring like the one in 
the proof of Theorem~\ref{th:false}.

Assume, by way of contradiction, that there are distinct monochromatic $e_1,\ldots,e_n$ such that
$$\sum_{i=1}^n a_ie_i = 0.$$

Let $\alpha_1,\ldots ,\alpha_n$ be such that
$e_i \in Z_{\alpha_i}$.
Since all of the elements in any $Z_\alpha$ are colored differently,
all of the $\alpha_i$'s are different.
We will assume $\alpha_1<\alpha_2<\cdots <\alpha_n$.
The other cases are similar.
Note that
$$e_n = -(1/a_n)\sum_{i=1}^{n-1} a_i e_i \in \CL(X)$$
and
$$e_1,\ldots,e_{n-1} \in Z_{\alpha_1}\cup \cdots \cup Z_{\alpha_{n-1}}\subseteq Y_{\alpha_{n-1}}.$$

Since $Y_{\alpha_{n-1}}=\CL(X_{\alpha_{n-1}})$
and $e_1,\ldots,e_{n-1}\in Y_{\alpha_{n-1}}$, we have
$e_n \in Y_{\alpha_{n-1}}$.
Hence $e_n\notin Z_{\alpha_n}$.  This is a contradiction.
\end{sketch}

\begin{note}
The converse to Theorem~\ref{th:gen} is not true.
The $s=2$ case of Theorem~\ref{th:fox} (in Section~\ref{se:moref}) 
states that
every $\aleph_0$-coloring of $\nat$ has a distinct monochromatic solution to
$x_1+2x_2=x_3+x_4+x_5$ is $\aleph_0$ iff $2^{\aleph_0} > \aleph_2$.
Therefore, assuming if $2^{\aleph_0}=\aleph_2$ (hence assuming CH is false) 
there is an $\aleph_0$-coloring of
$\nat$ such that there is no monochromatic distinct solution to $x_1+2x_2=x_3+x_4+x_5$.
This contradicts the converse of Theorem~\ref{th:gen}
\end{note}

\section{$\neg$ CH $\implies S$}

\begin{theorem}\label{th:true}
Assume $\neg$CH.
For any $\aleph_0$-coloring of $\real$
there exist distinct monochromatic $e_1,e_2,e_3,e_4$ such that
$$e_1+e_2=e_3+e_4.$$
\end{theorem}

\begin{proof}
By Fact~\ref{fa:omega}
there is an injection of $\omega_2$ into $\real$.
If $\alpha \in \omega_2$, then $x_\alpha$ is the real associated to it.

Given an $\aleph_0$-coloring $\COL$ of $\real$
we show that there exist distinct monochromatic $e_1,e_2,e_3,e_4$
such that $e_1+e_2=e_3+e_4$.

We define a map $F$ from
$\omega_2$ to $\omega_1\times\omega_1\times\omega_1\times\omega$ as follows:
\begin{enumerate}
\item
Let $\beta\in \omega_2$.
\item
Define a map from $\omega_1$ to $\omega$ by
$$\alpha\mapsto \COL(x_\alpha+x_\beta).$$
\item
Let $\alpha_1,\alpha_2,\alpha_3\in\omega_1$ be distinct elements of $\omega_1$, 
and $i\in\omega$, such that $\alpha_1,\alpha_2,\alpha_3$ all map to $i$.
Such $\alpha_1,\alpha_2,\alpha_3,i$ clearly exist since $\aleph_0 +\aleph_0 = \aleph_0 <\aleph_1$.
(There are $\aleph_1$ many elements that map to the same element of
$\omega$, but we do not need that.)
\item
Map $\beta$ to $(\alpha_1,\alpha_2,\alpha_3,i)$.
\end{enumerate}

Since $F$ maps a set of cardinality $\aleph_2$ to a set of
cardinality $\aleph_1$, there exists some element that is
mapped to twice by $F$ (actually there is an element that is
mapped to $\aleph_2$ times, but we do not need this).
Let $\alpha_1,\alpha_2,\alpha_3$ be distinct elements of $\omega_1$,
$i\in \omega$, and $\beta,\betap$ be distinct elements of $\omega_2$, such that
$$F(\beta)=F(\betap)=(\alpha_1,\alpha_2,\alpha_3,i).$$

Choose distinct $\alpha,\alphap\in\{\alpha_1,\alpha_2,\alpha_3\}$ such that $x_\alpha - x_\alphap \notin \{x_\beta - x_\betap,x_\betap - x_\beta\}$.  We can do this because there are at least three possible values for $x_\alpha - x_\alphap$.

Since $F(\beta)=(\alpha_1,\alpha_2,\alpha_3,i)$, we have
$$\COL(x_\alpha+x_\beta)=\COL(x_\alphap+x_{\beta})=i.$$

Since $F(\betap)=(\alpha_1,\alpha_2,\alpha_3,i)$, we have
$$\COL(x_\alpha+x_\betap)=\COL(x_\alphap+x_\betap)=i.$$

Let
\begin{eqnarray*}
e_1&=&x_\alpha+x_\beta\\
e_2&=&x_\alphap+x_\betap\\
e_3&=&x_\alphap+x_\beta\\
e_4&=&x_\alpha+x_\betap.
\end{eqnarray*}

Then
$$\COL(e_1)=\COL(e_2)=\COL(e_3)=\COL(e_4)=i$$
and
$$e_1+e_2=e_3+e_4=x_\alpha+x_\alphap+x_\beta+x_\betap.$$

Since $x_\alpha \ne x_\alphap$ and $x_\beta \ne x_\betap$, we have $\{e_1,e_2\} \cap \{e_3,e_4\} = \es$.

Moreover, the equation $e_1=e_2$ is equivalent to
$$x_\alpha - x_\alphap = x_\betap - x_\beta,$$
which is ruled out by our choice of $\alpha,\alphap$, and so $e_1\ne e_2$.

Similarly, $e_3 \ne e_4$.

Thus $e_1,e_2,e_3,e_4$ are all distinct.
\end{proof}

\section{A Generalization}\label{se:gen}

Recall that, for all $k\ge 1$, $\real$ and $\real^k$ are isomorphic as vector spaces over $\rat$.
Hence all the results of the last two sections about $\aleph_0$-colorings of $\real$ hold for $\real^k$.
In this more geometrical context, $e_1,e_2,e_3,e_4$ are vectors in $k$-dimensional Euclidean space, and the equation $e_1+e_2=e_3+e_4$ says that $e_1,e_2,e_3,e_4$ are the vertices of a parallelogram (whose area may be zero).  In particular, we have the following two theorems:

\begin{theorem}\label{th:gen2}
Fix any integer $k\ge 1$.  The following are equivalent:
\begin{enumerate}
\item $2^{\aleph_0} > \aleph_1$.
\item For any $\aleph_0$-coloring of $\real^k$, there exist distinct monochromatic vectors $e_1,e_2,e_3,e_4 \in \real^k$ such that $e_1+e_2 = e_3+e_4$.
\end{enumerate}
\end{theorem}

\begin{theorem}\label{th:gen3}
Fix any integers $k\ge 1$ and $n\ge 2$, and let $a_1,\ldots,,a_n\in\real$ be nonzero.  
Assume CH.
Then there exists
an $\aleph_0$-coloring of $\real^k$ such that there are no distinct monochromatic vectors $e_1,\ldots,e_n\in\real^k$ such that
\[ \sum_{i=1}^n a_i e_i = 0. \]
\end{theorem}

\section{More is Known: The Original Rado's Theorem}\label{se:moref}

Theorem~\ref{th:four} is a special case of a general theorem
about colorings and equations.

\begin{definition}\label{de:reg}~
Let $\vec b = (b_1,\ldots,b_n)\in \Z^n$
\begin{enumerate}
\item
$\vec b$ is {\it regular}
if the following holds:
{\it For all finite colorings of $\nat^+$
there exist
monochromatic $e_1,\ldots,e_n\in \nat^+$ such that
$$\sum_{i=1}^n b_ie_i = 0.$$
}
\item
$\vec b$ is {\it distinct regular}
if the following holds:
{\it For all finite colorings of 
$\nat^+$  there exist
monochromatic $e_1,\ldots,e_n\in \nat^+$, all distinct,
such that
$$\sum_{i=1}^n b_ie_i = 0.$$
}
\end{enumerate}
\end{definition}

In 1916 Schur~\cite{Schur} (see also \cite{VDWbook,GRS}) 
proved that, for any finite coloring of $\nat^+$, there is a monochromatic
solution to $x+y=z$. Using the above terminology he proved that
$(1,1,-1)$ was regular. 
For him this was a Lemma en route to an alternative proof to the following theorem
of Dickson~\cite{Dicksonlower}:

{\it For all $n\ge 2$ there is a prime $p_0$ such that, for all primes $p\ge p_0$,
 $x^n+y^n = z^n$ has a nontrivial solution mod $p$. }

For an English version of Schur's proof of Dickson's theorem see either
the book by Graham-Rothschild-Spencer~\cite{GRS} or the free online book
by Gasarch-Kruskal-Parrish~\cite{VDWbook}.

Schur's student Rado~\cite{radogerman,radoenglish} (see also \cite{VDWbook,GRS})
proved the following generalization of Schur's lemma:

\begin{theorem}\label{th:rado}~
\begin{enumerate}
\item
$\vec b$ is regular iff
some subset of $b_1,\ldots,b_n$ sums to 0.
\item
$\vec b$ is distinct-regular iff
some subset of $b_1,\ldots,b_n$ sums to 0
and there exists a vector $\vec \lambda$ of distinct
reals such that $\vec b \cdot \vec \lambda = 0$.
\end{enumerate}
\end{theorem}

\begin{note}~
\begin{enumerate}
\item
Rado's theorem is about any finite coloring. What about any (say) 3-coloring?
An equation is $k$-regular if for any $k$-coloring of $\nat$ there is a
monochromatic solution. There is no known characterization of
which equations are $k$-regular. 
Alexeev and Tsimmerman~\cite{AlexeevTs} have shown that
there are equations that are $(k-1)$-regular that are not $k$-regular.
\item
Theorem~\ref{th:rado} is Rado's theorem for single equations. There is a version
for sets of linear equations which you can find in~\cite{GRS,VDWbook}.
\end{enumerate}
\end{note}

We want to summarize the equivalence of $S$ and $\neg CH$ using the notion of
regularity.

\begin{definition}
$\vec b$ is {\it $\aleph_0$-distinct regular}
if the following holds:
{\it For all $\aleph_0$-colorings of $\real$
there exist
distinct monochromatic $e_1,\ldots,e_n\in \real$ such that
\begin{equation}\label{eqn:regular}
\sum_{i=1}^n b_ie_i = 0.
\end{equation}
}
\end{definition}

\begin{notation}
We may also say that an equation is $\aleph_0$-distinct-regular.
For example, the statement
$x_1+x_2=x_3+x_4$ is $\aleph_0$-distinct-regular means that
$(1,1,-1,-1)$ is $\aleph_0$-distinct-regular.
\end{notation}

If we combine Theorems~\ref{th:false} and \ref{th:true} and use this definition of regular
we obtain the following.
\begin{theorem}\label{th:reg}
$x_1+x_2=x_3+x_4$ is $\aleph_0$-distinct-regular iff $2^{\aleph_0} > \aleph_1$.
\end{theorem}

What about other linear equations over the reals?
Jacob Fox~\cite{foxrado} has generalized Theorem~\ref{th:reg} to prove the following.

\begin{theorem}\label{th:fox}
Let $s\in\nat$.  The equation
\begin{equation}\label{eqn:fox}
x_1+sx_2=x_3+\cdots+x_{s+3}
\end{equation}
is $\aleph_0$-distinct regular iff $2^{\aleph_0}> \aleph_s$.
\end{theorem}

%

\section{More is Known: Work Over a Field}

In Definition~\ref{de:reg} we defined when a tuple of integers is regular.
If we are concerned with coloring a field $F$ then we can easily define what
it means for a tuple of elements of a $F$ to be regular (or distinct-regular).

We state and prove a theorem of Ceder~\cite[Theorem~4]{Ceder:countable} in a way that gives us information about when 
$(b_1,b_2,b_3)$ is $\aleph_0$-distinct regular for $b_1,b_2,b_3$ in any field
We uses no assumptions outside of ZFC.
The proof is essentially Ceder's.

\begin{theorem}\label{th:field}
Let $F$ be any field.  For any $\gamma \in F - \{0,1\}$, there exists an $\aleph_0$-coloring of $F$ such that there are no distinct monochromatic $x,y,z\in F$ such that
\begin{equation}\label{eqn:gen-triangle}
z - x = \gamma(y - x).
\end{equation}
\end{theorem}

\begin{proof}
Let $K$ be some countable subfield of $F$ containing $\gamma$.  Choose a basis $\{b_i\}_{i\in I}$ of $F$ as a vector space over $K$, where $I$ is some index set with 
a linear order $<$.  
Then for any $w\in F$, there are unique coordinates $\{w_i\}_{i\in I}$ where each $w_i\in K$, only finitely many of the $w_i$ are nonzero, and
\[ w = \sum_{i\in I} w_i b_i. \]
Define the \emph{support} of $w$ as
\[ \support(w) := \{ i\in I \st w_i \ne 0 \} = \{i_1 < i_2 < \cdots < i_k\} \]
for some $k$.  Then define the \emph{signature} of $w$ as the $k$-tuple of the nonzero coordinates of $w$, namely

\[ \sg(w) := (w_{i_1},w_{i_2},\ldots,w_{i_k}). \]

Note that $\support(w)$ and $\sg(w)$ together uniquely determine $w$.  
Also note that there are only countably many possible signatures.
This is the key to how we define our $\aleph_0$-coloring:

$$COL(w) = \sg(w).$$

We will not use the notation $COL$ since we have $\sg$.

Assume, by way of contradiction, that $x,y,z\in F$ are distinct, satisfying Equation~(\ref{eqn:gen-triangle}),
such that $\sg(x)=\sg(y)=\sg(z)$.
Equation~(\ref{eqn:gen-triangle}) is equivalent to
\[ z = \gamma y + (1-\gamma)x, \]
or equivalently, since $\gamma\in K$,
\[ (\forall i\in I) [ \;z_i = \gamma y_i + (1-\gamma)x_i\;]. \]
Since $\sg(x) = \sg(y)$ and $x \ne y$, we must have $\support(x) \ne \support(y)$.  
Let $\ell\in I$ be the least element of $\support(x) \mathop{\triangle}\support(y)$.  Then for every $j<\ell$, we have $x_j = y_j$, and so
\[ z_j = \gamma y_j + (1-\gamma) x_j = y_j = x_j. \]
We now have two cases for $\ell$:
\begin{description}
\item[Case 1: $\ell \in \support(y)$.]  Then $y_\ell\ne 0$ and $x_\ell = 0$, because $\ell\notin\support(x)$.  This gives
\[ z_\ell = \gamma y_\ell \notin \{0,y_\ell\}, \]
which puts $\ell$ into $\support(z)$ and forces $\sg(z) \ne \sg(y)$. Contradiction.
\item[Case 2: $\ell \in \support(x)$.]  A similar argument, swapping the roles of $x$ and $y$ and swapping $\gamma$ with $1-\gamma$, shows that $\sg(z) \ne \sg(x)$. Contradiction.
\end{description}
\end{proof}

\begin{corollary}
Let $F$ be any field.  For any $b_1,b_2,b_3\in F$ not all zero, if $b_1+b_2+b_3 = 0$, then $(b_1,b_2,b_3)$ is not $\aleph_0$-distinct regular.
\end{corollary}

\begin{proof}
If $b_3=0$ then $b_2+b_3=0$ so $b_2=-b_3$. In this case we need to show that $(b_2,-b_2)$ is not $\aleph_0$-distinct-regular.
that is, we must show that there is a finite coloring of $F$ such that $b_2x=b_2y$ has no monochromatic
solution with $x\ne y$. Since any solution implies $x=y$ any coloring will suffice.
By similar reasoning we can dispense with the case where any of $b_2$ or $b_3$ is 0.

We want an $\aleph_0$-coloring of $F$ where there is no monochromatic distinct solution to

$$b_1e_1+b_2e_2+b_3e_3  = 0.$$

Dividing by $b_3$ and rearranging we obtain
\begin{equation}\label{eqn:triangle}
e_3 - e_1 = \gamma (e_2 - e_1),
\end{equation}

\noindent
where $\gamma := - b_2/b_3$.  Note that $\gamma\notin \{0,1\}$.
The desired $\aleph_0$-coloring of $F$ exists by Theorem~\ref{th:field}.
\end{proof}

\section{Is the Statement Really Natural?}


Theorem~\ref{th:false} and \ref{th:true} are stated as though they are about $\real$.
However, all that is used about $\real$ is that it is a vector space over $\rat$.
Hence the proof we gave really proves Theorem~\ref{th:sfb} below, from which Theorems~\ref{th:false} and \ref{th:true} (as well as Theorems~\ref{th:gen2} and \ref{th:gen3}, for that matter) follow as easy corollaries.

\begin{definition}
For any vector space $V$ over $\rat$, let
$S(V)$ be the statement,

{\it For any $\aleph_0$-coloring of $V$ there exist distinct monochromatic $e_1,e_2,e_3,e_4\in V$ such that $e_1+e_2=e_3+e_4$.}

\end{definition}

\begin{theorem}\label{th:sfb}
If $V$ is a vector space over $\rat$, then $S(V)$ iff $V$ has dimension
at least $\aleph_2$.
\end{theorem}

The proof of Theorem~\ref{th:sfb} is in ZFC.

One can ask the following: Since the result, when abstracted, has nothing to do with $\real$
and is just a statement provable in ZFC,
do we really have a \emph{natural} statement that is independent of ZFC?  We believe so.  

After you know that every finite coloring of $\nat$ has a distinct monochromatic solution to
$e_1+e_2=e_3+e_4$, it is natural to consider the following question:

\bigskip

{\it Does every $\aleph_0$-coloring of $\real$ have a distinct monochromatic solution to
$e_1+e_2=e_3+e_4$?}

\bigskip

This question can be understood by a bright high school or secondary school student with no knowledge of vector spaces.  The fact that \emph{after} you show that this question is independent of ZFC
you \emph{can then} abstract the proof to obtain Theorem~\ref{th:sfb} does not 
diminish the naturalness of the original question.


\section{Acknowledgments}

We would like to thank Jacob Fox for references and for
writing the paper that pointed us to this material.


\end{document}